\documentclass{amsart}
\usepackage{amssymb}
\usepackage{graphicx}

\newtheorem{theorem}{Theorem}
\newtheorem{cor}{Corollary}
\newtheorem{lemma}{Lemma}

\newtheorem{rem}{Remark}
\newtheorem{conj}{Conjecture}

\newcommand{\es}{{\mathcal S}}

\newcommand{\D}{{\mathbb D}}

\newcommand{\real}{{\operatorname{Re}\,}}

\def\be{\begin{equation}}
\def\ee{\end{equation}}

\newcommand{\bee}{\begin{enumerate}}
\newcommand{\eee}{\end{enumerate}}

\newcommand{\blem}{\begin{lem}}
\newcommand{\elem}{\end{lem}}
\newcommand{\bthm}{\begin{thm}}
\newcommand{\ethm}{\end{thm}}
\newcommand{\bcor}{\begin{cor}}
\newcommand{\ecor}{\end{cor}}
\newcommand{\beg}{\begin{example}}
\newcommand{\eeg}{\end{example}}
\newcommand{\begs}{\begin{examples}}
\newcommand{\eegs}{\end{examples}}
\newcommand{\bdefe}{\begin{defin}}
\newcommand{\edefe}{\end{defin}}
\newcommand{\bprob}{\begin{prob}}
\newcommand{\eprob}{\end{prob}}
\newcommand{\bei}{\begin{itemize}}
\newcommand{\eei}{\end{itemize}}

\newcommand{\bcon}{\begin{conj}}
\newcommand{\econ}{\end{conj}}
\newcommand{\bcons}{\begin{conjs}}
\newcommand{\econs}{\end{conjs}}
\newcommand{\bprop}{\begin{propo}}
\newcommand{\eprop}{\end{propo}}
\newcommand{\br}{\begin{rem}}
\newcommand{\er}{\end{rem}}
\newcommand{\brs}{\begin{rems}}
\newcommand{\ers}{\end{rems}}
\newcommand{\bo}{\begin{obser}}
\newcommand{\eo}{\end{obser}}
\newcommand{\bos}{\begin{obsers}}
\newcommand{\eos}{\end{obsers}}
\newcommand{\bpf}{\begin{pf}}
\newcommand{\epf}{\end{pf}}
\newcommand{\ba}{\begin{array}}
\newcommand{\ea}{\end{array}}
\newcommand{\beq}{\begin{eqnarray}}
\newcommand{\beqq}{\begin{eqnarray*}}
\newcommand{\eeq}{\end{eqnarray}}
\newcommand{\eeqq}{\end{eqnarray*}}

\title[Different approach to certain problems]{Different approach to certain problems in univalent function theory}

\author[M. Obradovi\'{c}]{Milutin Obradovi\'{c}}
\address{Department of Mathematics,
Faculty of Civil Engineering, University of Belgrade,
Bulevar Kralja Aleksandra 73, 11000, Belgrade, Serbia}
\email{obrad@grf.bg.ac.rs}

\author[N. Tuneski]{Nikola Tuneski}
\address{Department of Mathematics and Informatics, Faculty of Mechanical Engineering, Ss. Cyril and Methodius
University in Skopje, Karpo\v{s} II b.b., 1000 Skopje, Republic of North Macedonia.}
\email{nikola.tuneski@mf.edu.mk}


\dedicatory{Dedicated to the 70$^{th}$ birthday anniversary of Teodor Bulboac\v{a}, our great friend and inspiring colleague.}

\begin{document}

\begin{abstract}
Using different approach we consider some problems for the class $\mathcal{S}_{ex}^{\star}$ defined by the condition $\frac{zf'(z)}{f(z)}\prec e^{\alpha z}$, $0<\alpha\leq1.$ Our method is more simple than that used by Das and Sarkar in their work: \textit{Sharp coefficient estimates for the exponential starlike class $S^\star_{ex}$} (\cite{DS26}).
\end{abstract}

\maketitle

\section{Introduction and preliminaries}

\medskip

As usual, let $\mathcal{A}$ be the class of functions $f$ analytic in the open unit disc $\D=\{z:|z|<1\}$ of the form
\be\label{eq1}
f(z)=z+a_2z^2+a_3z^3+\cdots,
\ee
and let $\mathcal{S}$, $\mathcal{S}\subset \mathcal{A}$, be its subclass consisting of functions that are univalent (one-to-one and onto) in $\D$.

\medskip

Further, let $\Omega$ denote the class of (Schwarz) functions $\omega$ analytic  in $\D$ with
\be\label{eq2}
\omega(z)=c_{1}z +c_{2}z^{2}+\cdots,
\ee
and
$|\omega(z)|<1,\, z\in\D$.

\medskip

In \cite{DS26} the authors considered the class $\mathcal{S}_{ex}^{\star}$ of functions normalized in $\D$ satisfying the condition
\[\frac{zf'(z)}{f(z)}\prec e^{\alpha z},\quad 0<\alpha\leq1,\]
i.e.,
\be\label{eq4}
\frac{zf'(z)}{f(z)}= e^{\alpha \omega(z)},\quad 0<\alpha\leq1, z\in \D,
\ee
where $\omega \in \Omega$. Since $\real e^{\alpha z}>0$ for $0<\alpha\leq1$, $z\in \D,$ the functions from $\mathcal{S}_{ex}^{\star}$ are starlike univalent on the unit disc.
If $f$ and $\omega$ are given by \eqref{eq1} and \eqref{eq2}, then by \eqref{eq4}, after comparing the coefficients,
for $f\in\mathcal{S}_{ex}^{\star}$ we obtain
\begin{equation}\label{eq5}
\begin{split}
a_{2}&=\alpha c_{1},\\
a_{3}&=\frac{1}{2}\alpha c_{2}+\frac{3}{4}\alpha^{2}c_{1}^{2},\\
a_{4}&=\frac{1}{3}\alpha c_{3}+\frac{5}{6}\alpha^{2}c_{1}c_{2}+\frac{17}{36}\alpha^{3}c_{1}^{3}.
\end{split}
\end{equation}

\medskip

For $f\in \mathcal{S}$ the logarithmic coefficients $\gamma_{n}$, $n=1,2,\ldots$, are defined by
\be\label{eq6}
F_f(z) := \log\frac{f(z)}{z}=2\sum_{n=1}^\infty \gamma_{n} z^{n}.
\ee

From the relations \eqref{eq1} and \eqref{eq6}, after  equating the coefficients for initial logarithmic
coefficients we obtain:
\begin{equation}\label{eq7}
\begin{split}
 \gamma_{1}&=\frac{a_{2}}{2},\\
 \gamma_{2}&=\frac{1}{2}\left(a_3-\frac{1}{2}a_{2}^{2}\right), \\
 \gamma_{3}&=\frac{1}{2}\left(a_4-a_2a_3+\frac13a_2^3\right).
\end{split}
\end{equation}

\medskip

Relatively little exact information is known about the coefficients. The natural conjecture $|\gamma_n|\le1/n$,
inspired by the Koebe function (whose logarithmic coefficients are $1/n$) is false even in order of magnitude (see
Duren \cite[Section 8.1]{duren}).
For the class $\es$ the sharp estimates of single logarithmic coefficients S are known only for $\gamma_1$ and
$\gamma_2$, namely,
\[|\gamma_1|\le1\quad\mbox{and}\quad |\gamma_2|\le \frac12+e^{-2}=0.635335\ldots,\]
and are unknown for $n\ge3$. For estimates of $|\gamma_{3}|$ and $|\gamma_{4}|$ we can find in \cite{106} and
\cite{114}.

\medskip

Further, due to the famous Koebe's 1/4 theorem, a function $f\in\es$ given by \eqref{eq1} has its inverse at least on
the disc with radius 1/4, which is also analytic on this disc, with an expansion
\[f^{-1}(w) = w+A_2w^2+A_3w^3+\cdots.\]
Using the identity $f(f^{-1}(w))=w$, and after comparing the appropriate coefficients, we obtain
\begin{equation}\label{eq9}
\begin{array}{l}
A_{2}=-a_{2}, \\
A _{3}=-a_{3}+2a_{2}^{2} , \\
A_{4}= -a_{4}+5a_{2}a_{3}-5a_{2}^{3}.
\end{array}
\end{equation}

\medskip

Respectively, the logarithmic coefficients $\Gamma_{n}$ of the inverse function of a function $f$ from $ \mathcal{S}$
are defined by:
\[F_{f^{-1}}(z):=\log\frac{f^{-1}(z)}{z}=2\sum_{n=1}^\infty \Gamma_n z^n.\]
Applying the relations \eqref{eq5}, \eqref{eq7} and \eqref{eq9}, we obtain
\begin{equation}\label{eq10}
\begin{split}
 \Gamma_{1}&=\frac{A_{2}}{2}=-\frac{1}{2}\alpha c_{1},\\
 \Gamma_{2}&=\frac{1}{2}\left(A_3-\frac{1}{2}A_{2}^{2}\right)=-\frac{1}{4}\alpha c_{2}+\frac{3}{8}\alpha^{2}c_{1}^{2},\\
 \Gamma_{3}&=\frac{1}{2}\left(A_4-A_2A_3+\frac13A_2^3\right)=
 -\frac{1}{6}\alpha c_{3}+\frac{7}{12}\alpha^{2}c_{1}c_{2}-\frac{29}{72}\alpha^{3}c_{1}^{3}.
\end{split}
\ee

\medskip

In their paper \cite {DS26} for the functions of the class $\mathcal{S}_{ex}^{\star}$, the authors give some results concerning the estimates of coefficients $\Gamma_{n}$, $n=1,2,3$, as well as estimation of $|\Gamma_{2}|-|\Gamma_{1}|$ and of the Hankel determinant $|H_{2,1}(F_{f^{-1}/2})|$. For obtaining the results they used a property of Caratheodory functions, i.e., of functions from the class $\mathcal{R}$ that are analytic in $\D$ and are of type $p(z)= 1+p_{1}z +p_{2}z^{2}+\cdots$ with $\real p(z)>0$, $z\in \D$. Although there exists one-to-one correspondence between classes $\Omega $ and $\mathcal{R}$ given by $\omega(z)=\frac{p(z)-1}{p(z)+1}$, we find that approach with properties of $\Omega $ class is simpler. In that sense, for our approach we need the next two lemmas from \cite{126}. More details and some previous results can be found in \cite{128}.

\begin{lemma}\label{lem-1}
If  $\omega\in \Omega $ is given by \eqref{eq2}, then
$$|c_1|\leq 1,\quad |c_2|\leq 1-|c_1|^2,\quad \mbox{and}\quad |c_3|\leq 1-|c_1|^2-\frac{|c_2|^2}{1+|c_1|}.$$
\end{lemma}

\begin{lemma}\label{lem-2}
If  $\omega\in \Omega $ is given by \eqref{eq2}and if $\Psi(\omega)=|c_{3}+\mu c_{1}c_{2}+\nu c_{1}^{3}|$ ($\mu$ and $\nu$ are real numbers),
then the following sharp estimate $|\Psi(\omega)|\leq \Phi(\mu,\nu)$ holds, where
\[\Phi(\mu,\nu) = \left\{
  \begin{array}{ll}
    1, & \hbox{$(\mu,\nu) \in D_{1} \cup D_{2} \cup \{(2,1)\}$;} \\
    |\nu|, & \hbox{$(\mu,\nu) \in\bigcup_{k=3}^{7} D_{k} $;}\\
    \frac{2}{3}(|\mu|+1)\left(\frac{|\mu|+1}{3(|\mu|+1+\nu)}\right)^{1/2},&\hbox{$(\mu,\nu) \in D_{8} \cup D_{9}$;}\\
    \frac{1}{3}\nu \left(\frac{\mu^{2}-4}{\mu^{2}-4\nu}\right)\left(\frac{\mu^{2}-4}{3(\nu-1)}\right)^{1/2},&
     \hbox{$(\mu,\nu) \in D_{10} \cup D_{11}\setminus\{(2,1)\}$;}\\
    \frac{2}{3}(|\mu|-1)\left(\frac{|\mu|-1}{3(|\mu|-1-\nu)}\right)^{1/2},&\hbox{$(\mu,\nu) \in D_{12}$}.
  \end{array}
\right.\]
and
\begin{eqnarray*}
D_{1}&=& \left\{(\mu,\nu): |\mu| \leq \frac{1}{2}, \, -1 \leq \nu \leq 1 \right\},\\
D_{2}&=& \left\{(\mu,\nu): \frac{1}{2} \leq |\mu| \leq 2, \, \frac{4}{27}(|\mu|+1)^3 - (|\mu|+1) \leq \nu \leq 1 \right\},\\
D_{3}&=& \left\{(\mu,\nu): |\mu|\leq\frac{1}{2},\, \nu\leq -1 \right\},\\
D_{4}&=& \left\{(\mu,\nu): |\mu|\geq\frac{1}{2}, \, \nu\leq-\frac{2}{3}(|\mu|+1) \right\},\\
D_{5}& =& \left\{(\mu,\nu):|\mu| \leq 2, \, \nu \geq 1 \right\},\\
D_{6}&=& \left\{(\mu,\nu): 2\leq|\mu| \leq 4, \, \nu \geq \frac{1}{12}(\mu^{2}+8) \right\},\\
D_{7}&=& \left\{(\mu,\nu): |\mu|\geq 4, \, \nu \geq \frac{2}{3}(|\mu|-1) \right\},\\
D_{8}&=& \left\{(\mu,\nu): \frac{1}{2}\leq|\mu| \leq 2, \, -\frac{2}{3}(|\mu|+1)\leq \nu \leq \frac{4}{27}(|\mu|+1)^3 - (|\mu|+1)\right\},\\
D_{9}&=& \left\{(\mu,\nu): |\mu| \geq 2,\, -\frac{2}{3}(|\mu|+1)\leq \nu \leq \frac{2|\mu|(|\mu|+1)}{\mu^{2}+2|\mu|+4}\right\},\\
D_{10}&=& \left\{(\mu,\nu): 2\leq|\mu| \leq 4, \, \frac{2|\mu|(|\mu|+1)}{\mu^{2}+2|\mu|+4}\leq \nu \leq
\frac{1}{12}(\mu^{2}+8) \right\},
\end{eqnarray*}

\begin{eqnarray*}
D_{11}&=& \left\{(\mu,\nu): |\mu|\geq 4, \, \frac{2|\mu|(|\mu|+1)}{\mu^{2}+2|\mu|+4}\leq \nu \le \frac{2|\mu|(|\mu|-1)}{\mu^{2}-2|\mu|+4}\right\},\\
D_{12}&=& \left\{(\mu,\nu): |\mu|\geq 4, \, \frac{2|\mu|(|\mu|-1)}{\mu^{2}-2|\mu|+4}\leq \nu\leq
\frac{2}{3}(|\mu|-1) \right\}.
\end{eqnarray*}
\end{lemma}

\section{Main results}

\medskip

In this section we will use the two lemmas given above to prove the same results as in \cite{DS26}. We will use the same notations.

\begin{theorem}\label{th1}Let $f\in \mathcal{S}_{ex}^{\star}$ and let  $\Gamma_{1}$, $\Gamma_{2}$, $\Gamma_{3}$ be the logarithmic coefficients of $f^{-1}$. Then
\begin{itemize}
  \item[($i$)]  $ |\Gamma_{1}|\leq \frac{1}{2}\alpha $,\\
  \item[($ii$)]  $|\Gamma_{2}|\leq \left\{
  \begin{array}{ll}
  \frac{\alpha}{4}, & 0<\alpha \leq \frac{2}{3}, \\[2mm]
  \frac{3}{8}\alpha^{2},& \frac{2}{3}\leq \alpha \leq 1,
   \end{array}
   \right.
$\\[2mm]
  \item[($iii$)]   $|\Gamma_{3}|\leq \left\{
  \begin{array}{ll}
  \frac{\alpha}{6},& 0<\alpha\leq \alpha_{\star}, \\[2mm]
  \frac{\alpha(7\alpha+2)}{18}\sqrt{\frac{2(7\alpha+2)}{29\alpha^{2}+42\alpha+12}},& \alpha_{\star}\leq \alpha\leq \alpha_{b},\\[2mm]
  \frac{29\alpha (49\alpha^{2}-16)^{3/2}}{2232\sqrt{29\alpha^{2}-12}},&\alpha_{b}\leq \alpha \leq\alpha_{c},\\[2mm]
  \frac{29}{72}\alpha^{3},& \alpha_{c}\leq \alpha \leq 1,
   \end{array}
   \right.$
\end{itemize}
where $\alpha_{\star}=0.5427\ldots $ is the unique positive root of the equation $686 \alpha^{3}+327\alpha^{2}
-210\alpha -92=0$, $\alpha_{b}=0.6791\ldots $ is the unique positive root of the equation $1421\alpha^{3}+812\alpha^{2}
-712\alpha -336=0$, and $\alpha_{c}=\sqrt{\frac{32}{67}}=0.69109\ldots $ is the unique positive root of the equation $67\alpha^{2}-32=0$.
All bounds are sharp for every $\alpha\in(0,1].$
\end{theorem}

\medskip

\begin{proof}
Using the relations \eqref{eq10} for $f\in \mathcal{S}_{ex}^{\star}$, $0<\alpha\leq1$, as well as Lemma \ref{lem-1}, we obtain the following.
\begin{itemize}
  \item[($i$)] First, $|\Gamma_{1}|=\left|-\frac{1}{2}\alpha c_{1}\right|\leq\frac{1}{2}\alpha $ with equality for $|c_{1}|=1$, i.e. for the function $f$ defined by
$\frac{z f'(z)}{f(z)}=e^{\alpha z}$.
  \item[($ii$)] Next, \[
\begin{split}
| \Gamma_{2}|&= \left|-\frac{1}{4}\alpha c_{2}+\frac{3}{8}\alpha^{2}c_{1}^{2}\right|\\
&\leq\frac{\alpha}{8}\left(2|c_{2}|+3\alpha |c_{1}|^{2}\right) \\
&\leq \frac{\alpha}{8}\left(2(1-|c_{1}|^{2})+3\alpha |c_{1}|^{2}\right)\\
& =\frac{\alpha}{8}\left(2+(3\alpha -2)|c_{1}|^{2}\right),
\end{split}
\]
and the result easily follows from the sign of $3\alpha -2$ and the fact that  $0\leq |c_{1}|\leq1$. Equality in the case $0<\alpha \leq \frac{2}{3}$ holds for $c_{1}=0$ and $|c_{2}|=1$, and in the second case $\frac{2}{3}\le\alpha \leq 1$ for $|c_{1}|=1$ and $c_{2}=0$, i.e, for
$\omega (z)=z^{2}$ and $\omega (z)=z$ in \eqref{eq4}, respectively.
  \item[($iii$)] From \eqref{eq10} we have
\begin{equation}\label{eq11}
|\Gamma_{3}|=\frac{\alpha}{6}\left|c_{3}-\frac{7}{2}\alpha c_{1}c_{2}+\frac{29}{12}\alpha^{2}c_{1}^{3}\right|,
\ee
so, we can apply Lemma \ref{lem-2} with $|\mu|=\frac{7}{2}\alpha$ and $\nu=\frac{29}{12}\alpha^{2}$. We note that $\nu>0$ for all $0<\alpha\leq1$ . In that sense we consider some particular cases of $|\mu|$ and $\nu$.
\begin{itemize}
  \item[($a$)] If $|\mu|\leq \frac{1}{2}$, or equivalently $0<\alpha\leq \frac{1}{7}$, then $\nu\leq\frac{29}{12}\cdot \frac{1}{49}<1$,
which by Lemma \ref{lem-2} implies $\Phi(\mu,\nu)=1,$ i.e., $|\Gamma_{3}|\leq\frac{\alpha}{6}$ (case $D_{1}$ in \eqref{eq11}).
  \item[($b$)] Now let $\frac{1}{2}\leq |\mu|\leq 2$  ($\Leftrightarrow \frac{1}{7}\leq\alpha\leq \frac{4}{7}$) and
$\frac{4}{27}(|\mu|+1)^3 - (|\mu|+1) \leq \nu $. These conditions imply $\nu\leq1$ and they are valid if
$686 \alpha^{3}+327\alpha^{2}-210\alpha -92\leq 0$, i.e., for $\frac{1}{7}\le\alpha\leq \alpha_{\star} $,
where $\alpha_{\star}$ is given in the statements of this theorem. So, $(\mu,\nu)\in D_{2}$ and Lemma \ref{lem-2} again brings
$\Phi(\mu,\nu)=1$ and $|\Gamma_{3}|\leq\frac{\alpha}{6}$.
\end{itemize}

From (a) and (b) we conclude that $|\Gamma_{3}|\leq\frac{\alpha}{6}$ for $0<\alpha\leq \alpha_{\star}$. We proceed with the remaining cases.
\begin{itemize}
  \item[($c$)] Further, $(\mu,\nu)\in D_{8}$, if $\alpha_{\star}\leq \alpha \leq \frac{4}{7} $ and then we get
\[
\begin{split}
\Phi(\mu,\nu) &=\frac{2}{3}(|\mu|+1)\left(\frac{|\mu|+1}{3(|\mu|+1+\nu)}\right)^{1/2}\\
&=\frac{\alpha(7\alpha+2)}{18}\sqrt{\frac{2(7\alpha+2)}{29\alpha^{2}+42\alpha+12}}.
\end{split}
\]
Also, $(\mu,\nu)\in D_{9}$  for $|\mu|\geq2$ ($\Leftrightarrow \frac{4}{7}\leq\alpha\leq 1) $ and
$\nu \leq \frac{2|\mu|(|\mu|+1)}{\mu^{2}+2|\mu|+4}$, which is equivalent to
$1421\alpha^{3}+812\alpha^{2}-712\alpha -336\leq 0$ ($\Leftrightarrow 0\leq \alpha\leq \alpha_{b} $),
where $\alpha_{b}$ is the positive root of appropriate equation. In summary for
 $\frac{4}{7}\leq \alpha\leq \alpha_{b}$ we have the same result as case $D_{8}$. It means that for
$\alpha_{\star}\leq \alpha\leq \alpha_{b}$ we get the statement of the theorem.
  \item[($d$)] The case $(\mu,\nu)\in D_{10}$ follows from  $2\leq|\mu|\leq4$ ($\Leftrightarrow \frac{4}{7}\leq\alpha\leq 1) $ and
$\frac{2|\mu|(|\mu|+1)}{\mu^{2}+2|\mu|+4}\leq \nu \leq\frac{1}{12}(\mu^{2}+8)$ ($\Leftrightarrow \alpha_{b}\leq \alpha \leq\alpha_{c})$, meaning that for $\alpha_{b}\leq \alpha \leq\alpha_{c}$ we have
\[
\begin{split}
\Phi(\mu,\nu) &=\frac{1}{3}\nu \left(\frac{\mu^{2}-4}{\mu^{2}-4\nu}\right)\left(\frac{\mu^{2}-4}{3(\nu-1)}\right)^{1/2}\\
&=\frac{29\alpha (49\alpha^{2}-16)^{3/2}}{2232\sqrt{29\alpha^{2}-12}}.
\end{split}
\]
\item[($e$)] The last remaining  case is $(\mu,\nu)\in D_{6}$ yields from  $2\leq|\mu|\leq4$ and $|\nu| \geq \frac{1}{12}(\mu^{2}+8)$, which is equivalent to
$\alpha_{c}\leq \alpha \leq 1$ and implies $\Phi(\mu,\nu)=\frac{29\alpha^{3}}{72}$.
\end{itemize}

We note that all the previous estimates are sharp and the extremal functions for all cases are given in \cite[p.135 and 136]{126}, i.e., are again defined by \eqref{eq4}. For the cases $(a)$ and $(b)$, $\omega(z)=z^3$; for the case $(c)$, $\omega(z)=z\frac{c_{1}-z}{1-c_{1}z}$ with $c_{1}=\sqrt{\frac{|\mu|+1}{3(|\mu|+1+\nu)}}$; for $(d)$, $\omega(z)=c_1z+c_2z^2+c_3z^3+\cdots$ with
\[ c_1= \sqrt{\frac{3\mu^2-2(\mu^2+2)\nu}{3(\nu-1)(4\nu-\mu^2)}}, \quad c_2 = (1-c_1^2)e^{i\phi}, \quad c_3 = -c_1c_2e^{i\phi},\]
where $\phi = \pm \arccos\frac{\mu[2(\mu^2+2)-(\mu^2+8)\nu]}{2[3\mu^2-2(\mu^2+2)\nu]}$; and finally, for the case $(e)$, $\omega(z)=z$.
\end{itemize}
\end{proof}

\begin{theorem}\label{th2}Let $0< \alpha \leq 1$ and $f\in \mathcal{S}_{ex}^{\star}$. Then for its inverse logarithmic coefficients we have
$$-\frac{\alpha}{2}\sqrt{\frac{2}{3\alpha+2}}\leq |\Gamma_{2}|-|\Gamma_{1}|\leq \frac{\alpha}{4}.$$
These inequalities are sharp for all $\alpha\in(0,1]$.
\end{theorem}

\begin{proof}
Using \eqref{eq10} we obtain
\begin{equation}\label{eq12}
\begin{split}
|\Gamma_{2}|-|\Gamma_{1}|&=\left|-\frac{1}{4}\alpha c_{2}+\frac{3}{8}\alpha^{2}c_{1}^{2}\right|-\frac{1}{2}\alpha |c_{1}|\\
&=\frac{\alpha}{8}\left(|2c_{2}-3\alpha c_{1}^{2}|-4|c_{1}|\right).
\end{split}
\ee
From \eqref{eq12} and Lemma \ref{lem-1} we have
\[
\begin{split}
|\Gamma_{2}|-|\Gamma_{1}|&\leq \frac{\alpha}{8}\left(2|c_{2}|+3\alpha |c_{1}|^{2}-4|c_{1}|\right)\\
&\leq\frac{\alpha}{8}\left(2(1-|c_{1}|^{2})+3\alpha |c_{1}|^{2}-4|c_{1}|\right)\\
&=\frac{\alpha}{8}\left[(3\alpha-2)|c_{1}|^{2}-4|c_{1}|+2\right]\\
&\leq\frac{\alpha}{4},
\end{split}
\]
since $ 0\leq |c_{1}|\leq1$ and $0< \alpha \leq 1$. Equality is valid for $c_{1}=0$ and $|c_{2}|=1$, i.e., for the function $f$ defined
for $\omega(z)=z^{2}$ in \eqref{eq4}.

\medskip

For the lower bound, we will use that for complex numbers $a$ and $b$ we have $|a\pm b|\geq ||a|-|b||$. So, from \eqref{eq12} we get
\begin{equation}\label{eq13}
|\Gamma_{2}|-|\Gamma_{1}|\geq\frac{\alpha}{8}\left(\left|2|c_{2}|-3\alpha |c_{1}|^{2}\right|-4|c_{1}|\right).
\ee

In the case $2|c_{2}|-3\alpha |c_{1}|^{2}\geq0$ we have $2(1-|c_{1}|^{2})\geq2|c_{2}|\geq3\alpha |c_{1}|^{2}$, which implies $0\leq |c_{1}|\leq\sqrt{\frac{2}{3\alpha+2}}$. For such $|c_{1}|$, from \eqref{eq13} we obtain
$$|\Gamma_{2}|-|\Gamma_{1}|\geq\frac{\alpha}{8}\left(2|c_{2}|-3\alpha |c_{1}|^{2}-4|c_{1}|\right)
\geq-\frac{\alpha}{2}|c_{1}|\geq-\frac{\alpha}{2}\sqrt{\frac{2}{3\alpha+2}}.$$

\medskip

In the case when $\left|2|c_{2}|-3\alpha |c_{1}|^{2}\right|=3\alpha|c_{1}|^{2}-2|c_{2}|>0$ we want to prove that
\[|\Gamma_{2}|-|\Gamma_{1}|\geq\frac{\alpha}{8}\left(3\alpha|c_{1}|^{2}-2|c_{2}|-4|c_{1}|\right)\geq -\frac{\alpha}{2}\sqrt{\frac{2}{3\alpha+2}},\]
when $\sqrt{\frac{2}{3\alpha+2}}\leq|c_{1}|\leq1$.
Since $|c_{2}|\leq 1-|c_{1}|^{2}$, the last will follow from
\[(3\alpha+2)|c_{1}|^{2}-2\geq 4\left(|c_{1}|-\sqrt{\frac{2}{3\alpha+2}}\right),\]
which is equivalent to
$$ \left(|c_{1}|-\sqrt{\frac{2}{3\alpha+2}}\right)\left((3\alpha+2)\left( |c_{1}|+\sqrt{\frac{2}{3\alpha+2}}\right)-4  \right)\geq0.$$
Finally, the last is true since $|c_{1}|\geq\sqrt{\frac{2}{3\alpha+2}}$ (by assumption) and
\[
\begin{split}
(3\alpha+2)\left( |c_{1}|+\sqrt{\frac{2}{3\alpha+2}}\right)-4 &\geq 2(3\alpha+2)\sqrt{\frac{2}{3\alpha+2}}-4 \\
&=2\sqrt{2(3\alpha+2)}-4>0.
\end{split}
\]

\medskip

The equality is valid for $c_{1}=\sqrt{\frac{2}{3\alpha+2}}$ and $|c_{2}|= 1-|c_{1}|^{2}=\frac{3\alpha}{3\alpha+2}$,
i.e., for $\omega(z)=z\frac{c_{1}+z}{1+c_{1}z}=c_{1}z+(1-c_{1}^{2})z^{2}+\cdots $ in \eqref{eq4}.
\end{proof}

\medskip

In the next theorem we consider the estimate of the second Hankel determinant of the type
\begin{equation}\label{eq16}
H_{2,1}(F_{f^{-1}}/2)=\Gamma_{1}\Gamma_{3}-\Gamma_{2}^{2}.
\ee

We note that there exist more general definition of the Hankel determinant, but for our needs it is enough the type given above.

\begin{theorem}\label{th3} Let $0< \alpha \leq 1$ and let $f\in \mathcal{S}_{ex}^{\star}$. Then we have the following inequalities for
the Hankel determinant given in \eqref{eq16}:
$$|H_{2,1}(F_{f^{-1}}/2)|\leq \left\{
  \begin{array}{ll}
  \frac{\alpha^{2}}{16}, & 0<\alpha \leq \frac{2}{5}, \\
  \frac{\alpha^{2}}{16}\left[1+\frac{(5\alpha-2)^{2}}{12+60\alpha-35\alpha^{2}}\right],& \frac{2}{5}\leq \alpha \leq 1.
   \end{array}
   \right.
$$
The first inequality is sharp.
\end{theorem}

\begin{proof}
Using \eqref{eq16} and \eqref{eq10}, after some calculations, we have
\[
\begin{split}
|H_{2,1}(F_{f^{-1}}/2)|&=\frac{\alpha^{2}}{12}\left|c_{1}c_{3}-\frac{5}{4}\alpha c_{1}^{2}c_{2}
+\frac{35}{48}\alpha^{2} c_{1}^{4}-\frac{3}{4}c_{2}^{2}\right| \\
&\leq \frac{\alpha^{2}}{12}\left(|c_{1}||c_{3}|+\frac{5}{4}\alpha |c_{1}|^{2}|c_{2}|
+\frac{35}{48}\alpha^{2} |c_{1}|^{4}+\frac{3}{4}|c_{2}|^{2}\right).
\end{split}
\]
If we apply the estimations of $|c_{3}|$ and $|c_{2}|$ from Lemma \ref{lem-1} we get
\[
\begin{split}
&\quad  |H_{2,1}(F_{f^{-1}}/2)|\\
&\leq \frac{\alpha^{2}}{12}\left[|c_{1}|\left(1-|c_{1}|^{2}-\frac{|c_{2}|^{2}}{1+|c_{1}|}\right)+\frac{5}{4}\alpha |c_{1}|^{2}|c_{2}|+\frac{35}{48}\alpha^{2} |c_{1}|^{4}+\frac{3}{4}|c_{2}|^{2}\right]\\
&=\frac{\alpha^{2}}{12}\left[|c_{1}|-|c_{1}|^{3}+\frac{1}{4}\cdot \frac{3-|c_{1}|}{1+|c_{1}|}|c_{2}|^{2}+
\frac{5}{4}\alpha |c_{1}|^{2}|c_{2}|+\frac{35}{48}\alpha^{2} |c_{1}|^{4}\right]\\
&\leq \frac{\alpha^{2}}{12}\left[|c_{1}|-|c_{1}|^{3}+\frac{1}{4}\cdot \frac{3-|c_{1}|}{1+|c_{1}|}(1-|c_{1}|^{2})^{2}+
\frac{5}{4}\alpha |c_{1}|^{2}(1-|c_{1}|^{2})+\frac{35}{48}\alpha^{2} |c_{1}|^{4}\right]\\
&=\frac{\alpha^{2}}{48}\left[ 3+(5\alpha-2)|c_{1}|^{2}-\frac{1}{12}(12+60\alpha-35\alpha^{2})|c_{1}|^{4}\right] \equiv I.
\end{split}
\]

The second inequality above holds since $\frac{1}{4}\cdot \frac{3-|c_{1}|}{1+|c_{1}|}>0$, i.e., $\frac{1}{4}\cdot \frac{3-|c_{1}|}{1+|c_{1}|}|c_2|^2$ is an increasing function of $|c_2|^2$.

\medskip

If $5\alpha-2\leq0$, i.e., if $0<\alpha \leq \frac{2}{5}$ and since  $12+60\alpha-35\alpha^{2}>0 $ for all
$\alpha\in (0,1]$, then $|H_{2,1}(F_{f^{-1}}/2)|\leq\frac{\alpha^{2}}{16}$. Equality is valid for $c_{1}=0$ and
$c_{2}=1$, i.e., for $\omega(z)=z^{2}$ in \eqref{eq4}.
\end{proof}

In the next theorem we consider estimates the Hermitian-Toeplitz determinant of third order:
\begin{equation}\label{eq17}
T_{3,1}(f)=2\real (a_{2}^{2}\overline{a_{3}})-2|a_{2}|^{2}-|a_{3}|^{2}+1.
\ee

\medskip

\begin{theorem}\label{th4}Let $0< \alpha \leq 1$ and $f\in \mathcal{S}_{ex}^{\star}$ is given by \eqref{eq1}. Then for
the Hermitian-Toeplitz determinant defined by \eqref{eq17} we have the next sharp estimates:
$$B_{L}(\alpha)\leq T_{3,1}(f)\leq 1,$$
where
$$ B_{L}(\alpha)=
\left\{
  \begin{array}{ll}
  1-2\alpha^{2}+\frac{15}{16}\alpha^{4}, & 0<\alpha \leq \alpha_{H}, \\
  1-\frac{1}{4}\alpha^{2}-\frac{\alpha^{2}(\alpha+6)^{2}}{4(15\alpha^{2}+4\alpha -4)},& \alpha_{H}\leq\alpha\leq1,
   \end{array}
   \right. $$
 and $\alpha_{H}=\frac{-1+\sqrt{241}}{15}=0.9683\ldots$.
\end{theorem}

\begin{proof}
Since $|a_{3}-a_{2}^{2}|^{2}=(a_{3}-a_{2}^{2})\overline{(a_{3}-a_{2}^{2})}=|a_{3}|^{2}+|a_{2}|^{4}-2\real(a_{2}^{2}\overline{a_{3}})$, then $$2\real(a_{2}^{2}\overline{a_{3}})=|a_{3}|^{2}+|a_{2}|^{4}-|a_{3}-a_{2}^{2}|^{2}.$$
Using this fact, from \eqref{eq17} we get
\begin{equation}\label{eq18}
T_{3,1}(f)=1-|a_{3}-a_{2}^{2}|^{2}+|a_{2}|^{4}-2|a_{2}|^{2}.
\ee
Further, since $|a_{2}|=\alpha|c_{1}|\leq1$, from \eqref{eq18} for the upper bound we have
$$T_{3,1}(f)=1-|a_{3}-a_{2}^{2}|^{2}-|a_{2}|^{2}(2-|a_{2}|^{2})\leq1.$$
Equality is valid for $a_{2}=a_{3}=0$, i.e., for $c_{1}=c_{2}=0 $ and $c_{3}=1$. It means that the extremal function is given by \eqref{eq4} with $\omega (z)=z^{3}$.

\medskip

For the lower bound, from \eqref{eq5} we obtain
\[
\begin{split}
|a_{3}-a_{2}^{2}|&=\left|\frac{1}{2}\alpha c_{2}-\frac{1}{4}\alpha^{2}c_{1}^{2}\right|\\
&\leq \frac{\alpha}{4}(2|c_{2}|+\alpha |c_{1}|^{2})\\
&\leq \frac{\alpha}{4}(2(1-|c_{1}|^{2})+\alpha |c_{1}|^{2})\\
&=\frac{\alpha}{4}(2-(2-\alpha)|c_{1}|^{2}).
\end{split}
\]
Using this and \eqref{eq18}, we have
\[
\begin{split}
T_{3,1}(f)&\geq 1-\frac{\alpha^{2}}{16}\left[2-(2-\alpha)|c_{1}|^{2}\right]^{2}+\alpha^{4}|c_{1}|^{4}-2\alpha^{2}|c_{1}|^{2}\\
&=1-\frac{\alpha^{2}}{4}-\frac{\alpha^{2}(\alpha+6)}{4}|c_{1}|^{2}+\frac{\alpha^{2}}{16}(15\alpha^{2}
+4\alpha-4)|c_{1}|^{4}.
\end{split}
\]
In the case $15\alpha^{2}+4\alpha-4\leq0 $, i.e., for $0<\alpha\leq \frac{2}{5}=0.4 $, and since $0 \leq|c_{1}|\leq 1$, we have
$$T_{3,1}(f)\geq 1-\frac{\alpha^{2}}{4}-\frac{\alpha^{2}(\alpha+6)}{4}+\frac{\alpha^{2}}{16}(15\alpha^{2}
+4\alpha-4)=1-2\alpha^{2}+\frac{15}{16}\alpha^{4}.$$
In case when $15\alpha^{2}+4\alpha-4>0 $, i.e., for $\frac{2}{5}<\alpha \leq 1$, we again receive an expression of $|c_{1}|^{2}$ and the lower bound from the statement of the theorem easily follows if we consider two cases: $\frac{2}{5}<\alpha \leq
\alpha_{H}=\frac{-1+\sqrt{241}}{15}$ and $\alpha_{H}\leq\alpha\leq1$.
The extremal function for the lower bound, for the first case is defined by $\omega(z)=z$ in \eqref{eq4}, i.e., the same as in the case  $ 0<\alpha\leq \frac{2}{5}$, and in the second case for $\omega(z)=z\frac{c_{1}-z}{1-c_{1}z}$, where $c_{1}=\sqrt{\frac{2(\alpha+6)}{15\alpha^{2}+4\alpha-4}} .$
\end{proof}

\medskip

In the last theorem we give the sharp estimates of the generalized Fekete-Szego functional defined
for $f\in\mathcal{S}$ in \cite{LP2024} (see also \cite{TMN}) by:
\begin{equation}\label{eq19}
F_{\lambda, \mu}(f):=|a_{3}-\lambda a_{2}^{2}|-\mu |a_{2}|,\quad \lambda \in \mathbb{C},\,\mu>0.
\ee

\medskip

\begin{theorem}\label{th5} Let $0< \alpha \leq 1$,  $\lambda \in \mathbb{C}$ and $\mu>0$. If $f\in \mathcal{S}_{ex}^{\star}$ is given by \eqref{eq1}, then
$$B_{L}\leq \left|a_{3}-\lambda a_{2}^{2}\right|-\mu |a_{2}|\leq B_{U},$$
where
$$ {B}_{U}=
\left\{
  \begin{array}{ll}
\frac{\alpha}{2},& \alpha |3-4\lambda|< 2+4\mu, \\
\frac{\alpha^{2}|3-4\lambda|-4\alpha\mu}{4}, &\alpha |3-4\lambda|\geq 2+4\mu,
\end{array}
\right. $$
and
$$ {B}_{L}=
\left\{
  \begin{array}{ll}
-\frac{\alpha}{4}(4\mu-\alpha|3-4\lambda|),& \alpha |3-4\lambda|\leq 2\mu-2 \\[2mm]
-\alpha \mu \sqrt{\frac{2}{\alpha |3-4\lambda|+2}},& \alpha |3-4\lambda|\geq 2\mu^{2}-2\\[2mm]
-\frac{\alpha}{2}-\frac{\alpha \mu^{2}}{\alpha |3-4\lambda|+2}, & 2\mu-2<\alpha |3-4\lambda|<2\mu^{2}-2 ,
\end{array}
\right. $$
whenever the corresponding region is nonempty. The given bounds are sharp.
\end{theorem}
\begin{proof}
From \eqref{eq19} and \eqref{eq5}, having in mind that $|c_{1}|\leq1$ implies $|c_{1}|^{2}\leq |c_{1}|$, after some calculations for the upper bound we get
\[
\begin{split}
F_{\lambda, \mu}(f)&=\frac{\alpha}{4}\left[|2c_{2}+(3-4\lambda)\alpha c_{1}^{2}|-4\mu |c_{1}|\right]\\
&\leq\frac{\alpha}{4}\left[2(1-|c_{1}|^{2})+\alpha|3-4\lambda||c_{1}|^{2}-4\mu|c_{1}|^{2}\right] \\
&=\frac{\alpha}{4}\left[2+ (\alpha |3-4\lambda|-4\mu -2)|c_{1}|^{2}\right]
\end{split}
\]
and the result depends of the sign of $\alpha |3-4\lambda|-4\mu -2$ and easily follows.

\medskip

Equality holds for the function $f$ defined by \eqref{eq4}, in the first case when $c_{1}=0$ and $|c_{2}|=1$, i.e., for $\omega(z)=z^{2}$, and in the second case when  $\omega(z)=z$.

\medskip

For the lower bound we use once again that $|a\pm b|\geq ||a|-|b||$, where $a,b$ are complex numbers.
Since, previously we obtained
$$F_{\lambda, \mu}(f)=\frac{\alpha}{4}\left(|2c_{2}+\alpha (3-4\lambda)c_{1}^{2}|-4\mu |c_{1}|\right),$$
now we have
\begin{equation}\label{eq20}
F_{\lambda, \mu}(f) \ge \frac{\alpha}{4}\left(\left| 2|c_{2}|-\alpha |3-4\lambda||c_{1}|^{2}\right|
-4\mu |c_{1}|\right).
\ee
If $2|c_{2}|-\alpha |3-4\lambda||c_{1}|^{2}\geq0$, then
$$2(1-|c_{1}|^{2})\geq 2|c_{2}|\geq \alpha |3-4\lambda||c_{1}|^{2}, $$
i.e.
$$ |c_{1}|\leq \sqrt{\frac{2}{\alpha |3-4\lambda|+2}}:=|c_{1}|_{0}.$$
So, for $0\leq|c_{1}|\leq|c_{1}|_{0}$ and from \eqref{eq20}, we obtain
$$F_{\lambda, \mu}(f)\geq -\alpha\mu|c_{1}|_{0}=-\alpha\mu \sqrt{\frac{2}{\alpha |3-4\lambda|+2}}.$$
Now, let $|c_{1}|_{0} \leq|c_{1}|\leq 1$. In the second case, from \eqref{eq20}, we have
$$F_{\lambda, \mu}(f)\geq \frac{\alpha}{4}\left(\alpha |3-4\lambda||c_{1}|^{2}-2|c_{2}|
-4\mu |c_{1}|\right) , $$
or, since $|c_{2}|\leq 1-|c_{}|^{2}$,
\[F_{\lambda, \mu}(f)\geq\frac{\alpha}{4}\left[(\alpha |3-4\lambda|+2)|c_{1}|^{2} -4\mu |c_{1}|-2\right].\]
Now we need to find minimum of the the function
\begin{equation}\label{eq22}
G_{\lambda, \mu}(|c_1|):=(\alpha |3-4\lambda|+2)|c_{1}|^{2}-4\mu |c_{1}|-2,\quad 0\leq|c_{1}|\leq1.
\ee
If  $|c_{1}|_{1}:=\frac{2\mu}{\alpha |3-4\lambda|+2}\leq1 $, i.e., if $\alpha |3-4\lambda|\geq 2\mu-2$, then $G_{\lambda, \mu}(|c_1|)$ has its minimum equal to
$$-\frac{4\mu^{2}}{\alpha |3-4\lambda|+2}-2.$$
But, necessary to be $|c_{1}|_{1}\geq|c_{1}|_{0}$, which is equivalent to $\alpha |3-4\lambda|\leq 2\mu^{2}-2.$
In  that sense if we combine these consideration with \eqref{eq22}, we obtain the two last results for $B_{L}$.

\medskip

For the sharpness, the extremal function $f$ is given by \eqref{eq4}, in the first case when $\omega(z)=z $ in \eqref{eq4}, and in the in two last cases, when
$\omega(z)=e^{-i\beta} z\frac{|c_{1}|_{0}-z}{1-|c_{1}|_{0}z}$ and $\omega(z)=e^{-i\beta}z\frac{|c_{1}|_{1}-z}{1-|c_{1}|_{1}z}$, where $\beta=\arg (3-4\lambda)$.

\medskip

Namely, we can suppose that $3-4\lambda \neq0$, which implies $3-4\lambda = |3-4\lambda|e^{i\beta}$. The equality in $|a+b|\ge ||a|-|b||$ is valid if $a$ and $b$ have opposite arguments. In our case $a=2c_2$ and $b=\alpha(3-4\lambda)c_1^2$. For example, if $\omega(z)=e^{-i\beta} z\frac{|c_{1}|_{0}-z}{1-|c_{1}|_{0}z} = e^{-i\beta}|c_1|_0 z - e^{-i\beta}(1-|c_1|_0^2)z^2+\cdots $, then $a=-2e^{-i\beta}(1-|c_1|_0^2)$, $b = \alpha|3-4\lambda| e^{i\beta} e^{-2i\beta}|c_1|_0^2 = \alpha|3-4\lambda| e^{-i\beta}|c_1|_0^2$, which implies the sharpness.
\end{proof}

\end{document}